\documentclass[11pt,a4paper, final, twoside]{article}
\usepackage{amsmath}
\usepackage{fancyhdr}
\usepackage{amsthm}
\usepackage{amsfonts}
\usepackage{amssymb}
\usepackage{amscd}
\usepackage{amsthm}
\usepackage{graphicx}
\usepackage{afterpage}
\usepackage[colorlinks=true, urlcolor=blue,  linkcolor=blue, citecolor=blue]{hyperref}

\newtheorem{theorem}{Theorem}

\newtheorem{corollary}[theorem]{Corollary}

\newtheorem{definition}[theorem]{Definition}
\newtheorem{example}[theorem]{Example}

\newtheorem{lemma}[theorem]{Lemma}

\newtheorem{proposition}[theorem]{Proposition}

\begin{document}
\hyphenpenalty=10000

\begin{center}
{\Large \textbf{A Kazhdan group with an infinite outer automorphism group }}\\[5mm]
{\large {Traian Preda  }\\[10mm]
}
\end{center}

\textbf{Abstract}. D. Kazhdan has introduced in 1967 the Property (T) for local
compact groups (see \cite{kaz}). In this article we prove that for $n \geq 3$ and $ m \in \mathbb{N}$ the group
$SL_n (\textbf{K})\ltimes\mathcal{M}_{n,m}(\textbf{K})$ is a Kazhdan group having the outer automorphism group infinite.

\footnote{\textsf{2010 Mathematics Subject Classification:} 22D10;22D45} 
\footnote{\textsf{Keywords:} Representations of topological groups;Kazhdan Property (T); Mautner's lemma; } 
 
\afterpage{
\fancyhead{} \fancyfoot{} 
\fancyhead[LE, RO]{\bf\thepage}
\fancyhead[LO]{\small A Kazhdan group with an infinite outer automorphism group}
\fancyhead[RE]{\small Traian Preda  }}

\begin{definition}(\cite{bhv})
Let $(\pi,\mathcal{H})$ be a unitary representation of a topological group G.

(i) For a subset Q of G and real number $\varepsilon > 0$, a vector $\xi \in \mathcal{H}$ is $(Q, \varepsilon)$-invariant if : 
\begin{center}
$sup_{x\in Q}||\pi (x)\xi - \xi || < \varepsilon ||\xi||.$
\end{center}

(ii) The representation $(\pi,\mathcal{H})$ almost has invariant vectors if it has 
$(Q, \varepsilon)$ - invariant vectors  for every compact subset Q of G and every
$\varepsilon > 0$. If this holds, we write  $1_G  \prec \pi$.

(iii) The representation $(\pi,\mathcal{H})$ has non - zero invariant  vectors  if there exists 
$\xi \neq 0$ in $\mathcal{H}$ such that $\pi(x)\xi = \xi$ for all g$\in$G. If this holds, we
write $1_G  \subset \pi$.
\end{definition}

\begin{definition}(\cite{kaz})
Let G be a topological group.

G has Kazhdan's Property (T), or is a Kazhdan group, if there exists a compact subset Q 
of G and $\varepsilon > 0$ such that, whenever a unitary representation $\pi$ of G has a
$(Q, \varepsilon)$ - invariant vector, then $\pi$ has a non-zero invariant vector.
\end{definition}

\begin{proposition}(\cite{bhv})
Let G be a topological group.The following statements are equivalent:

(i) G has Kazhdan's Property(T);

(ii) whenever a unitary representation  $(\pi ,\mathcal{H})$ of G   weakly contains $1_G$, it
contains $1_G$ ( in symbols: $1_G \prec \pi$ implies $1_G \subset \pi$ ). 
\end{proposition}

\begin{definition}
Let \textbf{K} be a field. An absolute value on \textbf{K} is a real - valued function 
$x\to |x|$ such that, for all x and y in \textbf{K}:

(i)  $ |x| \geq 0$ and $|x| = 0  \Leftrightarrow x = 0$

(ii)$ |xy| = |x||y| $

(iii)$ | x + y | \leq |x|+|y| $.

An absolute value defines a topology on \textbf{K} given by the metric
\begin{center}
d(x, y) =$ | x - y|$.
\end{center}
\end{definition}

\begin{definition}
A field \textbf{K} is a local field if \textbf{K} can be equipped with an absolute value for
which \textbf{K} is locally compact and not discrete.
\end{definition}

\begin{example}
\textbf{K} = $ \mathbb{R}$ and \textbf{K} = $ \mathbb{C}$ with the usual absolute value are local fields.
\end{example}

\begin{example} (\cite{bhv} and \cite{hv}) Groups with Property (T):

a) Compact groups, $SL_n(\mathbb{Z} )$ for $n\geq 3$.

b) $SL_n(\textbf{K} )$ for $n\geq 3$ and \textbf{K} a local field.
\end{example}

\begin{lemma} ( Mautner's lemma)(\cite{bhv})

Let G be a topological group, and let $(\pi,\mathcal{H})$ be a unitary representation of G.
Let $x \in G$ and assume that there exists a net $(y_i)_i$ in G such that $\displaystyle \lim_{i}  y_ixy_i^{-1} = e$.
If $\xi$ is a vector in $\mathcal{H}$ which is fixed by $y_i$ for all i, then $\xi$ is fixed by x.
\end{lemma}

\begin{theorem}
Let \textbf{K} be a local field. The group $SL_n(\textbf{K})$ acts on $\mathcal{M}_{n,m}(\textbf{K})$ by left multiplication $(g, A) \to gA$, $g \in SL_n(\textbf{K})$  and  $A \in 
\mathcal{M}_{n,m}(\textbf{K})$.

Then the semi - direct product $SL_n(\textbf{K}) \ltimes \mathcal{M}_{n,m}(\textbf{K})$ has
Property (T) for $(\forall) n \geq 3$  and  $(\forall) m \in \mathbb{N}$.
\end{theorem}

\begin{proof}
Let $(\pi,\mathcal{H})$ be a unitary representation of G = $SL_n(\textbf{K}) \ltimes \mathcal{M}_{n,m}(\textbf{K})$ almost having invariant vectors. Since $SL_n(\textbf{K})$ has
Property (T), there exists a non - zero vector $\xi \in \mathcal{H} $ which is 
$SL_n (\textbf{K})$ - invariant.

Since $\textbf{K}$ is non - discret, there exists a net $(\lambda _i)_i$ in $\textbf{K}$ with
$\lambda _i \neq 0$ and such that $\displaystyle \lim_{i}  \lambda _i = 0$.

Let $\Delta_{pq} (x) \in \mathcal{M}_{n,m}(\textbf{K})$ the matrix with $x$ as (p,q) - entry and 
0 elsewhere and  $ (A_i)_{\alpha \beta} \in SL_n(\textbf{K})$ the matrix:
 \begin{equation} (A_i)_{\alpha , \beta} = 
\begin {cases}
 \lambda_i  &   if ~\alpha = \beta~  and ~\alpha= p \\[3mm]
\lambda_i^{-1} &  if ~\alpha = \beta~  and ~\alpha= (p+1)mod(n+1)+[p/n] \\[3mm]
1 & if ~\alpha = \beta~  and ~\alpha \notin \{p, (p+1)mod(n+1)+[p/n] \}\\[3mm]
0 & if ~\alpha \neq \beta
\end{cases}
\end{equation}
$\Rightarrow$ $A_i \Delta_{pq}(x) = \delta_{pq}(\lambda_i x)$, where
$\delta_{pq}(\lambda_ix)\in \mathcal{M}_{n,m}(\textbf{K})$ is the matrix with 
$\lambda_ix$ as (p,~q) - entry and 0 elsewhere.

Then $\displaystyle\lim_i A_i\Delta_{pq}(x) = 0_{n,m}$.

Since in G we have
\begin{center}
$(A_i, 0_{n,m})(I_n, \Delta_{pq}(x))( A_i , 0_{n,m})^{-1} = ( I_n , A_i \Delta_{pq}(x))$
\end{center}
and since $\xi \in \mathcal{H}$ is $( A_i , 0_{n,m})$ - invariant $\Rightarrow$

$\Rightarrow$ from Mautner's Lemma that $\xi$ is $\Delta_{pq}(x)$ - invariant.

Since $\Delta_{pq}(x)$ generates the group $ \mathcal{M}_{n,m}(\textbf{K})$ $\Rightarrow$
$\xi$ is G - invariant and G has Property (T).
\end{proof}

\begin{corollary}
The groups $SL_n(\textbf{K})\ltimes\textbf{K}^n$ and 
$SL_n(\mathbb{R}) \ltimes \mathcal{M}_n(\mathbb{R})$ has Property (T), 
$(\forall )n\geq 3$.
\end{corollary}

\begin{proposition}
For $\delta \in SL_n(\mathbb{Z})$, let $ S_{\delta}:\Gamma \to \Gamma $,
$S_{\delta}((\alpha, A)) = (\alpha, A\delta), (\forall ) (\alpha,A)\in\Gamma.$
Then: 

a) $S_{\delta}\in Aut(\Gamma).$

b) $\Phi : SL_n(\mathbb{Z}) \to Aut(\Gamma)$ , $\Phi(\delta ) = S_{\delta}$ is a group 
homomorphism.

c)$S_{\delta}\in Int (\Gamma)$ if and only if $\delta \in \{ \pm I\}$. In particular, the outer automorphism of $\Gamma$ is infinit.
\end{proposition}

\begin{proof}
a) $S_{\delta} ((\alpha_1,A_1)\cdot(\alpha_2, A_2)) = S_{\delta} ((\alpha_1,A_1))\cdot
S_{\delta} ((\alpha_2,A_2)) \Leftrightarrow$

$\Leftrightarrow S_{\delta} ((\alpha_1\alpha_2 , A_1 + \alpha_1 A_2)) = (\alpha_1, A_1\delta)\cdot(\alpha_2,
A_2\delta) \Leftrightarrow$

$\Leftrightarrow (\alpha_1\alpha_2 , (A_1 + \alpha_1 A_2)\delta) = (\alpha_1\alpha_2, A_1\delta + \alpha_1 A_2\delta) $

Analogous $S_{\delta^{-1}} $ is morfism and $S_{\delta}\cdot S_{\delta^{-1}} =
S_{\delta^{-1}} \cdot S_{\delta} = I_{\Gamma}. $

b)$\Phi (\delta_1 \cdot \delta_2 ) = \Phi (\delta_1)\cdot \Phi (\delta_2) \Leftrightarrow
S_{\delta_1 \cdot \delta_2} = S_{\delta_1}\cdot S_{\delta_2}.$

c) Assume that $S_{\delta} \in Int(\Gamma) \Rightarrow (\exists ) (\alpha_0 , A_0 )\in \Gamma$
such that

 $S_{\delta}((\alpha , A)) = (\alpha_0 , A_0)(\alpha , A)(\alpha_0, A_0)^{-1}, (\forall)
(\alpha,A)\in \Gamma.$

$\Rightarrow (\alpha , A\delta) = (\alpha_0\alpha\alpha_0^{-1} , A_0+\alpha_0 A - \alpha_0\alpha\alpha_0^{-1}A_0) \Rightarrow$

$\Rightarrow$ i) $\alpha =\alpha_0 \alpha \alpha_0^{-1} , (\forall ) \alpha \in SL_n(\mathbb{Z}) 
\Rightarrow \alpha \in \{ \pm I_n\}$

$\Rightarrow$ ii) $ A\delta = A_0 \pm A - \alpha A_0, (\forall) \alpha \in SL_n(\mathbb{Z}), (\forall ) A\in \mathcal{M}_n(\mathbb{Z}) \Rightarrow A_0 = 0_n$ and $\delta = \pm I_n$.
\[
\Rightarrow Out(\Gamma) =\raisebox{.2em}{$Aut(\Gamma)$}\left/\raisebox{-.2em}{$Int(\Gamma)$}\right.~~ is ~~infinite.
\]
\end{proof}

{\footnotesize University of Bucharest, Romania}
 
{\footnotesize e-mail: traianpr@yahoo.com}

\end{document}